\documentclass[a4paper,11pt,leqno]{amsart}
\usepackage{amsfonts}
\usepackage{amssymb}
\usepackage{amsmath}
\usepackage{graphicx}
\usepackage{epstopdf}

\newcommand{\R}{\mathbb R}
\newcommand{\N}{\mathbb N}

\newcommand{\C}{\mathbb C}
\newcommand{\Z}{\mathbb Z}

\newtheorem{theorem}{Theorem}[section]

\newtheorem{rem}{Remark}[section]

\input{tcilatex}

\newcommand{\RE}{\Re\it{e}}
\newcommand{\IM}{\Im\it{m}}

\begin{document}
\title{}

\markboth{Pierre Godard}  {Squares From Any Quadrilateral}

\begin{tabular}{l}
INTERNATIONAL JOURNAL OF GEOMETRY \\ 
Vol. 5 (2016), No. 1, 60 - 71%
\end{tabular}
\FRAME{itbpF}{0.5993in}{0.6054in}{0.2214in}{}{}{Graphic2}{\special{language
"Scientific Word";type "GRAPHIC";display "USEDEF";valid_file "F";width
0.5993in;height 0.6054in;depth 0.2214in;original-width
1.9043in;original-height 1.8896in;cropleft "0";croptop "1";cropright
"1";cropbottom "0";filename 'Graphic2';file-properties "XNPEU";}}

\centerline{}\bigskip

\centerline{}\bigskip

\centerline {\Large{\bf SQUARES FROM ANY QUADRILATERAL}}\centerline{}%
\centerline {\Large{\bf }}

\bigskip 

\begin{center}
{\large PIERRE GODARD}

\centerline{}
\end{center}

\bigskip

\textbf{Abstract.} In (Publ. Math. IHES, S88:43—46, 1998) A. Connes proposed an algebraic proof of Morley's trisector theorem.
He observed that the points of intersection of the trisectors are the fixed points of pairwise products of rotations around vertices of the triangle with angles two thirds of the corresponding angles of the triangle.
This paper enquires for similar results when the initial polygon is an arbitrary quadrilateral.
First we show that, when correctly gathered, fixed points of products of rotations around vertices of the quadrilateral with angles $(2n+1)/2$ of the corresponding angles of the quadrilateral form essentially six parallelograms for any integer $n$.
Several congruence relations are exhibited between these parallelograms.
Then, we show that if the original quadrilateral is itself a parallelogram, then for any integer $n$ four of the resulting parallelograms are squares.
Hence we present a function which, when applied twice, gives squares from any quadrilateral.
The proofs use only algebraic methods of undergraduate level.

\bigskip

\section{Notations and background}

In around 1899, F. Morley proved a theorem of Euclidean geometry that now bears his name: ``In any triangle, the intersections of adjacent trisectors form the vertices of an equilateral triangle".
This triangle is called Morley's triangle.
In 1998, A. Connes published a proof of Morley's theorem ``as a group theoretic property of the action of the affine group on the line" \cite{aConnes98}.
More precisely, for any triangle with vertices $a_1$, $a_2$, $a_3$ and angles $\hat{a}_1$, $\hat{a}_2$, $\hat{a}_3$, he defined $g_1$ as the rotation about $a_1$ through the angle $2\hat{a}_1/3$, and similarly for $g_2$ and $g_3$.
Then, from $g_1^3g_2^3g_3^3=1$ he deduced $\mathrm{fix}(g_1g_2)+j\mathrm{fix}(g_2g_3)+j^2\mathrm{fix}(g_3g_1)=0$, where $j$ is a non-trivial cubic root of unity and $\mathrm{fix}(g_1g_2)$ is the fixed point of the rotation $g_1g_2$, \textit{etc.}
This equation shows that the points $\mathrm{fix}(g_1g_2)$, $\mathrm{fix}(g_2g_3)$ and $\mathrm{fix}(g_3g_1)$
form the vertices of an equilateral triangle, which is Morley's triangle.
Inspired by Connes method, we exhibit identity relations between rotations about the vertices of a quadrilateral through angles proportional to the angles at these vertices, and show how fixed points of products of these rotations define parallelograms or squares.

--------------------------------------

\textbf{Keywords and phrases: }Euclidean plane geometry, parallelogram, square, Morley's triangles, rotation

\textbf{(2010)Mathematics Subject Classification: }51M04, 51M15

Received: 13.12.2015. \ In revised form: 07.04.2016. \ Accepted:
11.04.2016.\bigskip

\bigskip

We use the isomorphism between the plane $\R^2$ and the complex line $\C$; the imaginary unit is denoted by $i$.
The real and imaginary parts of $a\in\C$ are denoted respectively by $\RE[a]$ and $\IM[a]$.
We denote by $z_a$ the coordinate of the point $a\in\C$.
Let $r_{\alpha}$ be the rotation about the origin $O$ through the angle $\alpha$, and let $t_{z_a}$ be the translation by the vector joining $O$ to the point $a$ with coordinate $z_a$:
\begin{equation*}
r_{\alpha}:\C\rightarrow\C,z\mapsto e^{i\alpha}z
\end{equation*}
and
\begin{equation*}
t_{z_a}:\C\rightarrow\C,z\mapsto z+z_a.
\end{equation*}
Then the rotation about the point $a$ through angle $\alpha$ is
\begin{equation*}
r_{a,\alpha}=t_{z_a}\circ r_{\alpha}\circ(t_{z_a})^{-1}:z\mapsto e^{i\alpha}(z-z_a)+z_a.
\end{equation*}
Composition of operations, denoted here by $\circ$, will not be explicitely written anymore.

The composition of two rotations is as follows:
\begin{equation*}
r_{a_1,\alpha_1}r_{a_2,\alpha_2}z=e^{i(\alpha_1+\alpha_2)}z+z_{a_1}(1-e^{i\alpha_1})+z_{a_2}e^{i\alpha_1}(1-e^{i\alpha_2})
\end{equation*}
of which we deduce that if $\alpha_1+\alpha_2\neq 0\ mod(2\pi)$, $r_{a_1,\alpha_1}r_{a_2,\alpha_2}$ is a rotation about the point with coordinate $\big(z_{a_1}(1-e^{i\alpha_1})+z_{a_2}e^{i\alpha_1}(1-e^{i\alpha_2})\big)/(1-e^{i(\alpha_1+\alpha_2)})$ through the angle $\alpha_1+\alpha_2$, and if $\alpha_1+\alpha_2=0\ mod(2\pi)$, $r_{a_1,\alpha_1}r_{a_2,\alpha_2}$ is a translation by $z_{a_1}(1-e^{i\alpha_1})+z_{a_2}(e^{i\alpha_1}-1)$.
In particular, if $\alpha_1=\alpha_2=\pi$, then $r_{a_1,\alpha_1}r_{a_2,\alpha_2}=t_{2(z_{a_1}-z_{a_2})}$. 

The notation $P=[a_1,a_2,a_3,a_4]$ means that $P$ is the oriented quadrilateral whose vertices are successively $a_1$, $a_2$, $a_3$, $a_4$, and edges $a_1a_2$, $a_2a_3$, $a_3a_4$, $a_4a_1$.
We denote by $P'$ the quadrilateral $P$ with the opposite orientation, that is $[a_1,a_2,a_3,a_4]'=[a_1,a_4,a_3,a_2]$.

Lastly, $\{i,j,k,l\}=\{1,2,3,4\}$ means that the tuple $(i,j,k,l)$ is a permutation of $(1,2,3,4)$.

In the next section we define, for any $n\in\Z$, points $b_{ijkl,n}$ whose coordinates are functions of the coordinates of the vertices of a quadrilateral $P$.
We show that these points group in parallelograms, and exhibit some properties of these parallelograms.
In the third section we show that if $P$ is itself a parallelogram, then the exhibited parallelograms are mostly squares.

\bigskip

\section{Parallelograms from any quadrilateral}

\begin{theorem}
\textit{Let $P$ be a quadrilateral and denote its four vertices by $a_1$, $a_2$, $a_3$ and $a_4$; for $i\in\{1,2,3,4\}$ let $\hat{a}_i$ be the angle at the vertex $a_i$, and for any $n\in\Z$ define $\alpha_{i,n}$ by $\alpha_{i,n}:=\frac{2n+1}{2}\hat{a}_i$.
Let $(i,j,k,l)$ be a permutation of $(1,2,3,4)$.
Finally, let $b_{ijkl,n}$ be the fixed point of the rotation $r_{a_i,\alpha_{i,n}}r_{a_j,\alpha_{j,n}}$ $r_{a_k,\alpha_{k,n}}r_{a_l,\alpha_{l,n}}$.
Then the following holds:
\begin{enumerate}
\item
the quadrilaterals $P_{ijkl,n}:=[b_{ijkl,n},b_{ijlk,n},b_{jilk,n},b_{jikl,n}]$ and $P_{ijkl,n}':=[b_{ijkl,n},b_{jikl,n},b_{jilk,n},b_{ijlk,n}]$ are parallelograms;
\item
$P_{ijkl,n}$ is congruent to $P_{klij,-n-1}$; More precisely, $b_{ijkl,n}=b_{lkji,-n-1}$, so that the rotation through $\pi$ about the center of $P_{ijkl,n}$ tranforms $P_{ijkl,n}$ into $P_{klij,-n-1}'$;
\item
$P_{ijkl,n}$ is congruent to $P_{klji,n}$; More precisely, let $R_{ij,n}$ be the rotation through the angle $-(\alpha_{i,n}+\alpha_{j,n})$ about the center $O_{ij,n}$ whose coordinate is
$$z_{O_{ij,n}}=\frac{1}{2(1-e^{i(\alpha_{i,n}+\alpha_{j,n})})}\Big(z_{a_i}(1-e^{i\alpha_{i,n}})(1+e^{i\alpha_{j,n}})+z_{a_j}(1-e^{i\alpha_{j,n}})(1+e^{i\alpha_{i,n}})\Big);$$
then $R_{ij,n}P_{ijkl,n}=P_{klji,n}'$;
\item
let $\mathcal{A}_Q$ be the area of a parallelogram $Q$.
Then $\mathcal{A}_{P_{ijkl,n}}=\mathcal{A}_{P_{ikjl,n}}+\mathcal{A}_{P_{ilkj,n}}$ and $\mathcal{A}_{P_{ijkl,n}}=\mathcal{A}_{P_{ijlk,-n-1}}$.
\end{enumerate}\label{th_parallelogram1}}
\end{theorem}

\begin{proof}
\textit{First, we note that $b_{ijkl,n}$ is well-defined since $\alpha_{i,n}+\alpha_{j,n}+\alpha_{k,n}+\alpha_{l,n}=(2n+1)\pi$ when $i$, $j$, $k$ and $l$ are all different.
We now prove each item separately.
\\
\begin{enumerate}
\item For convenience, for $i\in\{1,2,3,4\}$, we write $r_i$ for $r_{a_i,\alpha_{i,n}}$.
Now if $\{i,j,k,l\}=\{1,2,3,4\}$ then $r_lr_kr_jr_i$ is the rotation about $b_{lkji,n}$ through the angle $(2n+1)\pi$.
Hence it is an involution.
Similarly, $r_ir_jr_kr_l$ is an involution.
Thus, we have
\begin{align*}
(r_ir_jr_kr_lr_ir_jr_lr_k)(r_jr_ir_lr_kr_jr_ir_kr_l)&=r_ir_jr_kr_lr_ir_j(r_lr_kr_jr_ir_lr_kr_jr_i)r_kr_l \\
&=r_ir_jr_kr_lr_ir_jr_kr_l \\
&=id
\end{align*}
where $id:z\mapsto z$.
On the other hand
\begin{align*}
&(r_ir_jr_kr_lr_ir_jr_lr_k)(r_jr_ir_lr_kr_jr_ir_kr_l) \\
=&(r_{b_{ijkl,n},(2n+1)\pi}r_{b_{ijlk,n},(2n+1)\pi})(r_{b_{jilk,n},(2n+1)\pi}r_{b_{jikl,n},(2n+1)\pi}) \\
=&t_{2(z_{b_{ijkl,n}}-z_{b_{ijlk,n}})}t_{2(z_{b_{jilk,n}}-z_{b_{jikl,n}})} \\
=&t_{2(z_{b_{ijkl,n}}-z_{b_{ijlk,n}}+z_{b_{jilk,n}}-z_{b_{jikl,n}})}.
\end{align*}
Hence $z_{b_{ijkl,n}}-z_{b_{ijlk,n}}+z_{b_{jilk,n}}-z_{b_{jikl,n}}=0$,
which means that the sides $b_{ijkl,n}b_{ijlk,n}$ and $b_{jikl,n}b_{jilk,n}$ have equal length and are parallel, a characteristic property for $P_{ijkl,n}=[b_{ijkl,n},b_{ijlk,n},b_{jilk,n},b_{jikl,n}]$ to be a parallelogram.
That $P_{ijkl,n}'$ is also a parallelogram is now obvious.
\item
We first compute that
\begin{align}\label{eq_z_b_ijkl}
&z_{b_{ijkl,n}}=\frac{1}{2}\Big(z_{a_i}(1-e^{i\alpha_{i,n}})+z_{a_j}e^{i\alpha_{i,n}}(1-e^{i\alpha_{j,n}}) \nonumber\\
&\qquad+z_{a_k}e^{i(\alpha_{i,n}+\alpha_{j,n})}(1-e^{i\alpha_{k,n}})+z_{a_l}e^{i(\alpha_{i,n}+\alpha_{j,n}+\alpha_{k,n})}(1-e^{i\alpha_{l,n}})\Big)
\end{align}
when $\{i,j,k,l\}=\{1,2,3,4\}$.
Then, we have
$$\alpha_{i,-n-1}=\frac{2(-n-1)+1}{2}\hat{a}_i=-\frac{2n+1}{2}\hat{a}_i=-\alpha_{i,n}.$$
From this, a straightforward computation shows that $z_{b_{ijkl,n}}=z_{b_{lkji,-n-1}}$.
Hence, denoting by $R_{ijkl,n}$ the rotation through $\pi$ about the center of $P_{ijkl,n}$ we have
\begin{align*}
R_{ijkl,n}P_{ijkl,n}&=R_{ijkl,n}[b_{ijkl,n},b_{ijlk,n},b_{jilk,n},b_{jikl,n}] \\
&=[b_{jilk,n},b_{jikl,n},b_{ijkl,n},b_{ijlk,n}] \\
&=[b_{klij,-n-1},b_{lkij,-n-1},b_{lkji,-n-1},b_{klji,-n-1}] \\
&=[b_{klij,-n-1},b_{klji,-n-1},b_{lkji,-n-1},b_{lkij,-n-1}]' \\
&=P_{klij,-n-1}'.
\end{align*}
\item 
A straightforward computation shows that $z_{R_{ij,n}b_{ijkl,n}}=z_{b_{klji,n}}$.
Hence
\begin{align*}
R_{ij,n}P_{ijkl,n}&=R_{ij,n}[b_{ijkl,n},b_{ijlk,n},b_{jilk,n},b_{jikl,n}] \\
&=[b_{klji,n},b_{lkji,n},b_{lkij,n},b_{klij,n}] \\
&=P_{klji,n}'.
\end{align*}
\item 
The area of the parallelogram $P_{ijkl,n}$ is given by
$$\mathcal{A}_{P_{ijkl,n}}=\IM[(z_{b_{jikl,n}}-z_{b_{ijkl,n}})\overline{(z_{b_{ijlk,n}}-z_{b_{ijkl,n}})}].$$
But, with the help of equation \ref{eq_z_b_ijkl}, we have
\begin{equation}\label{eq_area1}
z_{b_{jikl,n}}-z_{b_{ijkl,n}}=\frac{1}{2}(1-e^{i\alpha_{i,n}})(1-e^{i\alpha_{j,n}})(z_{a_j}-z_{a_i})\end{equation}
and
\begin{equation}\label{eq_area2}
z_{b_{ijlk,n}}-z_{b_{ijkl,n}}=\frac{1}{2}e^{i(\alpha_{i,n}+\alpha_{j,n})}(1-e^{i\alpha_{k,n}})(1-e^{i\alpha_{l,n}})(z_{a_l}-z_{a_k}).
\end{equation}
Hence,
\begin{align*}
\mathcal{A}_{P_{ijkl,n}}=&\frac{1}{4}\IM[(1-e^{i\alpha_{i,n}})(1-e^{i\alpha_{j,n}})(z_{a_j}-z_{a_i})\cdots  \\
&\quad\overline{e^{i(\alpha_{i,n}+\alpha_{j,n})}(1-e^{i\alpha_{k,n}})(1-e^{i\alpha_{l,n}})(z_{a_l}-z_{a_k})}] \\
=&\frac{1}{4}\IM[(1-e^{-i\alpha_{i,n}})(1-e^{-i\alpha_{j,n}})(1-e^{-i\alpha_{k,n}})\cdots  \\
&\quad(1-e^{-i\alpha_{l,n}})(z_{a_j}-z_{a_i})\overline{(z_{a_l}-z_{a_k})}].
\end{align*}
Now using $\alpha_{l,n}+\alpha_{i,n}+\alpha_{j,n}+\alpha_{k,n}=(2n+1)\pi$ one can readily show that
\begin{align*}
I:&=(1-e^{-i\alpha_{i,n}})(1-e^{-i\alpha_{j,n}})(1-e^{-i\alpha_{k,n}})(1-e^{-i\alpha_{l,n}}) \\
&=2i\Big(\sin(\alpha_{i,n})+\sin(\alpha_{j,n})+\sin(\alpha_{k,n})-\sin(\alpha_{i,n}+\alpha_{j,n})\cdots \\
&\quad-\sin(\alpha_{i,n}+\alpha_{k,n})-\sin(\alpha_{j,n}+\alpha_{k,n})+\sin(\alpha_{i,n}+\alpha_{j,n}+\alpha_{k,n})\Big).
\end{align*}
It is thus a purely imaginary number.
Therefore, we have
$$\mathcal{A}_{P_{ijkl,n}}=\frac{1}{4}\IM[I]\RE[(z_{a_j}-z_{a_i})\overline{(z_{a_l}-z_{a_k})}].$$
By a permutation of the indices we obtain
\begin{align*}
&\mathcal{A}_{P_{ikjl,n}}=\frac{1}{4}\IM[I]\RE[(z_{a_k}-z_{a_i})\overline{(z_{a_l}-z_{a_j})}] \\
&\mathcal{A}_{P_{ilkj,n}}=\frac{1}{4}\IM[I]\RE[(z_{a_l}-z_{a_i})\overline{(z_{a_j}-z_{a_k})}].
\end{align*}
Hence $\mathcal{A}_{P_{ijkl,n}}-\mathcal{A}_{P_{ikjl,n}}-\mathcal{A}_{P_{ilkj,n}}$ is proportional to
$$\RE[(z_{a_j}-z_{a_i})\overline{(z_{a_l}-z_{a_k})}-(z_{a_k}-z_{a_i})\overline{(z_{a_l}-z_{a_j})}-(z_{a_l}-z_{a_i})\overline{(z_{a_j}-z_{a_k})}]$$
which is easily seen to vanish.
\\
From $R_{ijkl,n}P_{ijkl,n}=P_{klij,-n-1}'$ the parallelograms $P_{ijkl,n}$ and $P_{klij,-n-1}$ are congruent but have opposite orientations, so $\mathcal{A}_{P_{ijkl,n}}=-\mathcal{A}_{P_{klij,-n-1}}$.
Then from $R_{ij,n}P_{ijkl,n}=P_{klji,n}'$ we have $R_{kl,-n-1}P_{klij,-n-1}=P_{ijlk,-n-1}'$ and so $\mathcal{A}_{P_{klij,-n-1}}=-\mathcal{A}_{P_{ijlk,-n-1}}$.
Thus $\mathcal{A}_{P_{ijkl,n}}=\mathcal{A}_{P_{ijlk,-n-1}}$.
\end{enumerate}}
\end{proof}

\begin{rem}
\textit{\label{remark0}
It is easily shown that $z_{O_{ij,n}}=\mu_{i,n}z_{a_i}+\mu_{j,n}z_{a_j}$ where
$$\mu_{i,n}=\frac{1-e^{i\alpha_{i,n}}+e^{i\alpha_{j,n}}-e^{i(\alpha_{i,n}+\alpha_{j,n})}}{2(1-e^{i(\alpha_{i,n}+\alpha_{j,n})})}$$
and
$$\mu_{j,n}=\frac{1-e^{i\alpha_{j,n}}+e^{i\alpha_{i,n}}-e^{i(\alpha_{i,n}+\alpha_{j,n})}}{2(1-e^{i(\alpha_{i,n}+\alpha_{j,n})})}.$$
Now, it is clear that $\mu_{i,n}+\mu_{j,n}=1$ so that, in $\R^2$, $O_{ij,n}$ belongs to the line joining $a_i$ to $a_j$.
\\
Moreover, from $\alpha_{i,-n-1}=-\alpha_{i,n}$, we deduce from the preceding expressions that $z_{O_{ij,-n-1}}=\overline{\mu_{i,n}}z_{a_i}+\overline{\mu_{j,n}}z_{a_j}$.
However it is easily shown that
$$\mu_{i,n}=\frac{1-\cos(\alpha_{i,n})+\cos(\alpha_{j,n})-\cos(\alpha_{i,n}+\alpha_{j,n})}{2\Big(1-\cos(\alpha_{i,n}+\alpha_{j,n})\Big)}$$
and similarly
$$\mu_{j,n}=\frac{1-\cos(\alpha_{j,n})+\cos(\alpha_{i,n})-\cos(\alpha_{i,n}+\alpha_{j,n})}{2\Big(1-\cos(\alpha_{i,n}+\alpha_{j,n})\Big)}$$
so that $\mu_{i,n}$ and $\mu_{j,n}$ are real.
Hence 
$$z_{O_{ij,-n-1}}=\mu_{i,n}z_{a_i}+\mu_{j,n}z_{a_j}=z_{O_{ij,n}}.$$
}
\end{rem}

In the next two remarks, $\{i,j,k,l\}=\{1,2,3,4\}$.

\begin{rem}
\textit{\label{remark1}
Point 1 of theorem \ref{th_parallelogram1} remains valid if we replace respectively $r_{a_i,\alpha_{i,n}}$, $r_{a_j,\alpha_{j,n}}$, $r_{a_k,\alpha_{k,n}}$, $r_{a_l,\alpha_{l,n}}$, by $r_{a_i,\alpha_{i,n}+m_i\pi/2}$, $r_{a_j,\alpha_{j,n}+m_j\pi/2}$, $r_{a_k,\alpha_{k,n}+m_k\pi/2}$, $r_{a_l,\alpha_{l,n}+m_l\pi/2}$ provided that $m_i+m_j+m_k+m_l$ is a multiple of 4.
In this remark, we restrict ourself to integer $m_i$'s.
To avoid repetition, we limit $m_i$ to $0\leq m_i\leq 3$ for $i\in\{1,2,3,4\}$.
Hence each parallelogram exhibited in theorem \ref{th_parallelogram1} gives in fact 64 parallelograms:
\begin{itemize}
\item 1 for $\{m_i,m_j,m_k,m_l\}=\{0,0,0,0\}$,
\item 1 for $\{m_i,m_j,m_k,m_l\}=\{1,1,1,1\}$,
\item 12 for $\{m_i,m_j,m_k,m_l\}=\{2,1,1,0\}$,
\item 6 for $\{m_i,m_j,m_k,m_l\}=\{2,2,0,0\}$,
\item 12 for $\{m_i,m_j,m_k,m_l\}=\{3,1,0,0\}$,
\item 12 for $\{m_i,m_j,m_k,m_l\}=\{3,3,2,0\}$,
\item 6 for $\{m_i,m_j,m_k,m_l\}=\{3,3,1,1\}$,
\item 12 for $\{m_i,m_j,m_k,m_l\}=\{3,2,2,1\}$,
\item 1 for $\{m_i,m_j,m_k,m_l\}=\{2,2,2,2\}$,
\item 1 for $\{m_i,m_j,m_k,m_l\}=\{3,3,3,3\}$.
\end{itemize}
These parallelograms are the analogues of the 27 Morley's triangles, see e.g. \cite{aDobbs38}.
Note however that not all Morley's triangles are regular.
}
\end{rem}

\begin{rem}
\textit{\label{remark2}
More generally, for any $M\in\N$ and $i\in\{1,2,3,4\}$, we can replace $r_{a_i,\alpha_{i,n}}$ by $r_{a_i,\alpha_{i,n}+m_i\pi/M}$ provided that $m_i+m_j+m_k+m_l$ is a multiple of $2M$.
Of course, other combinations are possible provided than the sum on the angles added on $r_1$, $r_2$, $r_3$ and $r_4$ is a multiple of $2\pi$.
}
\end{rem}

For each $n\in\N$, theorem \ref{th_parallelogram1} exhibits 48 parallelograms.
Nevertheless, only six of them have different vertex sets, \textit{e.g.} $P_{1234,n}$, $P_{3412,n}$, $P_{1324,n}$, $P_{2413,n}$, $P_{1432,n}$ and $P_{3214,n}$.
The figure \ref{fig_parallelogram1} presents an example of a quadrilateral $P$ and these six associated parallelograms when $n=0$.
The parallelograms $P_{1234,0}$, $P_{3412,0}$, $P_{1432,0}$ and $P_{3214,0}$ are positively oriented, whereas $P_{1324,0}$ and $P_{2413,0}$ are negatively oriented and have thus a negative area.
The figure \ref{fig_parallelogram2} presents the same quadrilateral $P$ and the six parallelograms $P_{1234,n}$ for $n\in\{0,1,2,3,4,5\}$.

\begin{figure}[h!t]
    \centering
    \includegraphics[width=8.4cm]{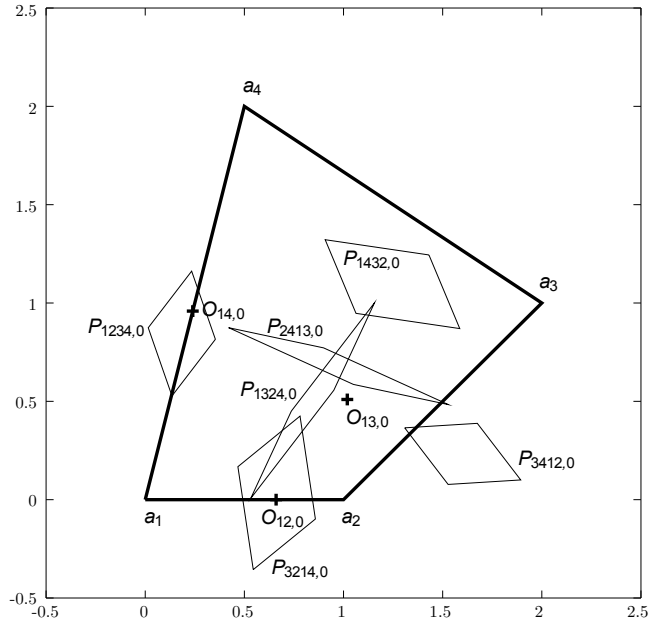}
    \caption{Example illustrating the theorem \ref{th_parallelogram1}. The coordinates of the vertices $a_1$, $a_2$, $a_3$ and $a_4$ of the quadrilateral $P$ are respectively
$(0,0)$, $(1,0)$, $(2,1)$ and $(1/2,2)$. The quadrilateral $P$, the parallelograms $P_{1234,0}$, $P_{3412,0}$, $P_{1324,0}$, $P_{2413,0}$, $P_{1432,0}$, $P_{3214,0}$ and the points $O_{12,0}$, $O_{13,0}$, $O_{14,0}$ are shown.}
    \label{fig_parallelogram1}
\end{figure}

\begin{figure}[h!t]
    \centering
    \includegraphics[width=8.4cm]{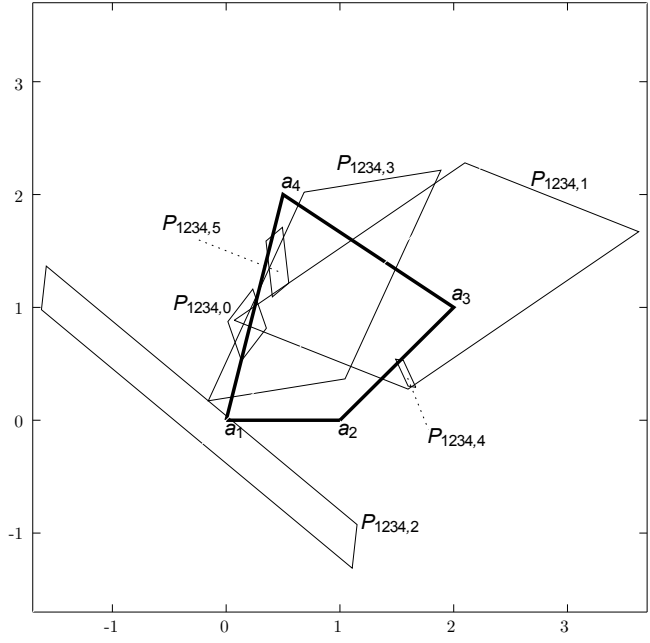}
    \caption{Example illustrating the theorem \ref{th_parallelogram1}. The coordinates of the vertices $a_1$, $a_2$, $a_3$ and $a_4$ of the quadrilateral $P$ are respectively
$(0,0)$, $(1,0)$, $(2,1)$ and $(1/2,2)$. The quadrilateral $P$ and the parallelograms $P_{1234,0}$, $P_{1234,1}$, $P_{1234,2}$, $P_{1234,3}$, $P_{1234,4}$, $P_{1234,5}$ are shown.}
    \label{fig_parallelogram2}
\end{figure}

\section{Squares from any parallelogram}
We now reconsider theorem \ref{th_parallelogram1} when $P$ is itself a parallelogram.

\begin{theorem}
\textit{\label{th_parallelogram2}
Let $P$ be a parallelogram and denote its four vertices by $a_1$, $a_2$, $a_3$ and $a_4$ and its center by $C$; for $i\in\{1,2,3,4\}$ let $\hat{a}_i$ be the angle at the vertex $a_i$, and for any $n\in\Z$ define $\alpha_{i,n}$ by $\alpha_{i,n}:=\frac{2n+1}{2}\hat{a}_i$.
Let $\mathcal{S}$ be the set composed of the tuples $(1,2,3,4)$, $(1,4,3,2)$ and of their cyclic permutations, and let $(i,j,k,l)$ be in $\mathcal{S}$.
Finally, let $b_{ijkl,n}$ be the fixed point of the rotation $r_{a_i,\alpha_{i,n}}r_{a_j,\alpha_{j,n}}r_{a_k,\alpha_{k,n}}r_{a_l,\alpha_{l,n}}$.
Then the following holds:
\begin{enumerate}
\item
the following quadrilaterals are squares
$$P_{ijkl,n}:=[b_{ijkl,n},b_{ijlk,n},b_{jilk,n},b_{jikl,n}]$$
$$P_{ijkl,n}':=[b_{ijkl,n},b_{jikl,n},b_{jilk,n},b_{ijlk,n}]$$
$$P_{ijlk,n}:=[b_{ijlk,n},b_{ijkl,n},b_{jikl,n},b_{jilk,n}]$$
$$P_{ijlk,n}':=[b_{ijlk,n},b_{jilk,n},b_{jikl,n},b_{ijkl,n}];$$
\item
$P_{ijkl,n}$ is congruent to $P_{klij,-n-1}$; More precisely, $b_{ijkl,n}=b_{lkji,-n-1}$, so that the rotation through $\pi$ about the center of $P_{ijkl,n}$ tranforms $P_{ijkl,n}$ into $P_{klij,-n-1}'$;
\item
$P_{ijkl,n}$ is congruent to $P_{klji,n}$; More precisely, let $R_{ij,n}$ be the rotation through the angle $-\frac{2n+1}{2}\pi$ about the center $O_{ij,n}$ whose coordinate is
$$z_{O_{ij,n}}=\frac{1-\cos(\alpha_{i,n})+(-1)^n\sin(\alpha_{i,n})}{2}z_{a_i}+\frac{1+\cos(\alpha_{i,n})-(-1)^n\sin(\alpha_{i,n})}{2}z_{a_j};$$
then $R_{ij,n}P_{ijkl,n}=P_{klji,n}'$;
\item
the rotation through $\pi$ about $C$ transforms $P_{ijkl,n}$ into $P_{klij,n}'$;
\item
let $C_{ijkl,n}$ be the center of the square $P_{ijkl,n}$.
Then all the centers $C_{ijkl,n}$ group into two lines, one for the even $n$, the other for the odd $n$; Moreover these lines are perpendicular to the side $a_ia_j$.
\item
the diagonal $b_{ijkl,n}b_{jilk,n}$ of $P_{ijkl,n}$ is parallel to $a_ia_j$;
\item
the quadrilateral $P_n:=[C_{ijkl,n},C_{kjil,n},C_{klij,n},C_{ilkj,n}]$ is a parallelogram;
Moreover, the center of $P_n$ coincides with the center $C$ of $P$.
\end{enumerate}
}
\end{theorem}

\begin{proof}
\textit{\begin{enumerate}
\item
If $P=[a_1,a_2,a_3,a_4]$ is a parallelogram and if $(i,j,k,l)$ is in $\mathcal{S}$ then
\begin{subequations}\label{cond_parallelogram}
\begin{align}
	&\alpha_{i,n}=\alpha_{k,n} \label{cond_parallelogram1}\\
	&\alpha_{j,n}=\alpha_{l,n} \label{cond_parallelogram2}\\
	&\alpha_{i,n}+\alpha_{j,n}=\frac{2n+1}{2}\pi \label{cond_parallelogram3}\\
	&z_{a_i}-z_{a_j}+z_{a_k}-z_{a_l}=0. \label{cond_parallelogram4}
\end{align}
\end{subequations}
\\
Now, from equations \ref{eq_area2}, \ref{cond_parallelogram} and \ref{eq_area1} we have successively
\begin{align*}
z_{b_{ijkl,n}}-z_{b_{ijlk,n}}&=\frac{1}{2}e^{i(\alpha_{i,n}+\alpha_{j,n})}(1-e^{i\alpha_{k,n}})(1-e^{i\alpha_{l,n}})(z_{a_k}-z_{a_l}) \\
&=\frac{-(-1)^ni}{2}(1-e^{i\alpha_{i,n}})(1-e^{i\alpha_{j,n}})(z_{a_i}-z_{a_j}) \\
&=-(-1)^ni(z_{b_{ijkl,n}}-z_{b_{jikl,n}}).
\end{align*}
The quadrilateral $P_{ijkl,n}:=[b_{ijkl,n},b_{ijlk,n},b_{jilk,n},b_{jikl,n}]$ has thus two adjacent edges of the same length and that form a right angle.
Since we know from theorem \ref{th_parallelogram1} that it is a parallelogram, it is a square.
Similarly, we show that
$$z_{b_{ijlk,n}}-z_{b_{ijkl,n}}=(-1)^ni(z_{b_{ijlk,n}}-z_{b_{jilk,n}})$$
so that $P_{ijlk,n}$ is a square.
That $P_{ijkl,n}'$ and $P_{ijlk,n}'$ are also squares is obvious.
\item
This is an immediate consequence of point 2 in theorem \ref{th_parallelogram1}.
\item
This is an immediate consequence of points 3 in theorem \ref{th_parallelogram1}, remark \ref{remark0} and equation \ref{cond_parallelogram3}.
\item
We show that for $(i,j,k,l)\in\mathcal{S}$, $C$ sits in the middle of the segment joining $b_{ijkl,n}$ to $b_{klij,n}$.
From equation \ref{eq_z_b_ijkl}, we have
\begin{align*}
&\frac{z_{b_{ijkl,n}}+z_{b_{klij,n}}}{2} \\
=&\frac{1}{4}\Big\{(1+e^{i(\alpha_{k,n}+\alpha_{l,n})})\Big(z_{a_i}(1-e^{i\alpha_{i,n}})+z_{a_j}e^{i\alpha_{i,n}}(1-e^{i\alpha_{j,n}})\Big) \\
&\qquad+(1+e^{i(\alpha_{i,n}+\alpha_{j,n})})\Big(z_{a_k}(1-e^{i\alpha_{k,n}})+z_{a_l}e^{i\alpha_{k,n}}(1-e^{i\alpha_{l,n}})\Big)\Big\}.
\end{align*}
Use of equations \ref{cond_parallelogram} leads easily to:
$$\frac{z_{b_{ijkl,n}}+z_{b_{klij,n}}}{2}=\frac{z_{a_i}+z_{a_k}}{2}$$
which is of course $z_C$.
This means that the rotation $r_{C,\pi}$ maps $b_{ijkl,n}$ into $b_{klij,n}$.
Therefore, we have
\begin{align*}
r_{C,\pi}P_{ijkl,n}&=[b_{klij,n},b_{lkij,n},b_{lkji,n},b_{klji,n}] \\
&=P_{klij,n}'.
\end{align*}
\item
When $(i,j,k,l)\in\mathcal{S}$ we obtain from equations \ref{eq_z_b_ijkl} and \ref{cond_parallelogram} that
\begin{align*}
C_{ijkl,n}&=\frac{z_{b_{ijkl,n}}+z_{b_{jilk,n}}}{2} \\
&=\frac{1}{4}\Big\{\Big(1-e^{i\alpha_{i,n}}+(-1)^nie^{-i\alpha_{i,n}}-(-1)^ni\Big)\Big(z_{a_i}+(-1)^niz_{a_k}\Big)\cdots \\
&\qquad+\Big(1+e^{i\alpha_{i,n}}-(-1)^nie^{-i\alpha_{i,n}}-(-1)^ni\Big)\Big(z_{a_j}+(-1)^niz_{a_l}\Big)\Big\} \\
\end{align*}
so that
\begin{align*}
&C_{ijkl,n}-C_{ijkl,n-2} \\
=&\frac{1}{4}\Big\{\Big(-e^{i\alpha_{i,n}}+e^{i\alpha_{i,n-2}}+(-1)^nie^{-i\alpha_{i,n}}-(-1)^nie^{-i\alpha_{i,n-2}}\Big)\Big(z_{a_i}+(-1)^niz_{a_k}\Big)\cdots \\
&\quad+\Big(e^{i\alpha_{i,n}}-e^{i\alpha_{i,n-2}}-(-1)^nie^{-i\alpha_{i,n}}+(-1)^nie^{-i\alpha_{i,n-2}}\Big)\Big(z_{a_j}+(-1)^niz_{a_l}\Big)\Big\}.
\end{align*}
But $e^{i\alpha_{i,n}}-e^{i\alpha_{i,n-2}}$ is easily seen to be $2i\sin(\hat{a}_i)e^{i\alpha_{i,n-1}}$, so that
\begin{align*}
&C_{ijkl,n}-C_{ijkl,n-2} \\
&=\frac{i\sin(\hat{a}_i)}{2}\Big\{\Big(-e^{i\alpha_{i,n-1}}-(-1)^nie^{-i\alpha_{i,n-1}}\Big)\Big(z_{a_i}+(-1)^niz_{a_k}\Big)\cdots \\
&\qquad+\Big(e^{i\alpha_{i,n-1}}+(-1)^nie^{-i\alpha_{i,n-1}}\Big)\Big(z_{a_j}+(-1)^niz_{a_l}\Big)\Big\}. \\
&=\frac{i\sin(\hat{a}_i)}{2}\Big(e^{i\alpha_{i,n-1}}+(-1)^nie^{-i\alpha_{i,n-1}}\Big)\Big(1-(-1)^ni\Big)(z_{a_j}-z_{a_i}) \\
&=i\sin(\hat{a}_i)\Big(\cos(\alpha_{i,n-1})+(-1)^n\sin(\alpha_{i,n-1})\Big)(z_{a_j}-z_{a_i}).
\end{align*}
Hence, the vector joining $C_{ijkl,n-2}$ to $C_{ijkl,n}$ is obtained by applying to the vector joining $a_i$ to $a_j$ a dilation by $\sin(\hat{a}_i)\big(\cos(\alpha_{i,n-1})+(-1)^n\sin(\alpha_{i,n-1})\big)$ followed by a rotation through $\pi/2$ about the origin.
In particular, all the $C_{ijkl,n}$ with $n$ even belong to a unique line, and this line is perpendicular to the line joining $a_i$ to $a_j$.
Similarly, all the $C_{ijkl,n}$ with $n$ odd belong to a unique line, and this line is perpendicular to the line joining $a_i$ to $a_j$.
\item
Uses of equations \ref{eq_z_b_ijkl} and \ref{cond_parallelogram} and straightforward manipulations lead to
\begin{align*}
z_{b_{ijkl,n}}-z_{b_{jilk,n}}=(z_{a_i}-z_{a_j})\Big(1-\cos(\alpha_{i,n})-(-1)^n\sin(\alpha_{i,n})\Big).
\end{align*}
Hence $z_{b_{ijkl,n}}-z_{b_{jilk,n}}$ is proportional with a real coefficient to $z_{a_i}-z_{a_j}$, so that the diagonal $b_{ijkl,n}b_{jilk,n}$ of $P_{ijkl,n}$ is parallel to $a_ia_j$.
\item
The center $C_{ijkl,n}$ of $P_{ijkl,n}$ has coordinate $z_{C_{ijkl,n}}=\frac{z_{b_{ijkl,n}}+z_{b_{jilk,n}}}{2}$.
From equation \ref{eq_z_b_ijkl}, this is equal to
\begin{align*}
z_{C_{ijkl,n}}&=\frac{1}{4}\Big\{z_{a_i}(1-e^{i\alpha_{i,n}})(1+e^{i\alpha_{j,n}})+z_{a_j}(1-e^{i\alpha_{j,n}})(1+e^{i\alpha_{i,n}})\cdots \\
&\quad+e^{i(\alpha_{i,n}+\alpha_{j,n})}\Big(z_{a_k}(1-e^{i\alpha_{k,n}})(1+e^{i\alpha_{l,n}})+z_{a_l}(1-e^{i\alpha_{l,n}})(1+e^{i\alpha_{k,n}})\Big)\Big\}
\end{align*}
so that
\begin{align*}
&z_{C_{ijkl,n}}+z_{C_{klij,n}} \\
=&\frac{1}{4}\Big\{(1+e^{i(\alpha_{k,n}+\alpha_{l,n})})\Big(z_{a_i}(1-e^{i\alpha_{i,n}})(1+e^{i\alpha_{j,n}})+z_{a_j}(1-e^{i\alpha_{j,n}})(1+e^{i\alpha_{i,n}})\Big)\cdots \\
&+(1+e^{i(\alpha_{i,n}+\alpha_{j,n})})\Big(z_{a_k}(1-e^{i\alpha_{k,n}})(1+e^{i\alpha_{l,n}})+z_{a_l}(1-e^{i\alpha_{l,n}})(1+e^{i\alpha_{k,n}})\Big)\Big\}.
\end{align*}
Now if $P$ is a parallelogram the relations \ref{cond_parallelogram} allow to simplify this expression:
$$z_{C_{ijkl,n}}+z_{C_{klij,n}}=z_{a_i}+z_{a_k}.$$
A permutation between $i$ and $k$ shows that $z_{C_{ijkl,n}}+z_{C_{klij,n}}=z_{C_{kjil,n}}+z_{C_{ilkj,n}}$, so that
$P_n:=[C_{ijkl,n},C_{kjil,n},C_{klij,n},C_{ilkj,n}]$ is a parallelogram.
As a byproduct, the center of $P_n$ coincides with the center of $P$.
\end{enumerate}
}
\end{proof}

\begin{rem}
\textit{\label{remark3}
The relations \ref{cond_parallelogram} are not satisfied by all of the 64 parallelograms exhibited in remark \ref{remark1}.
Indeed, $\alpha_{i,n}+m_i\pi/2+\alpha_{j,n}+m_j\pi/2=\pi/2\,mod(\pi)$ if and only if $m_i+m_j$ is even, and $\alpha_{i,n}=\alpha_{k,n}$ and $\alpha_{j,n}=\alpha_{l,n}$ if and only if $m_i=m_k$ and $m_j=m_l$, respectively.
Thus, any square exhibited in theorem \ref{th_parallelogram2} gives in fact 8 squares for
\begin{align*}
(m_i,m_j,m_k,m_l)\in\{&(0,0,0,0),(1,1,1,1),(2,2,2,2),(3,3,3,3),\cdots \\
&(0,2,0,2),(2,0,2,0),(1,3,1,3),(3,1,3,1)\}.
\end{align*}}
\end{rem}

For each $n\in\N$, theorem \ref{th_parallelogram2} exhibits 32 squares.
Nevertheless, only four of them have different vertex sets, \textit{e.g.}  $P_{1234,n}$, $P_{3214,n}$, $P_{3412,n}$, $P_{1432,n}$.
The figure \ref{fig_square1} presents an example of a parallelogram $P$ and these four squares when $n=0$.
The figure \ref{fig_square2} presents the same parallelogram $P$ and the six squares $P_{1234,n}$ for $n\in\{0,1,2,3,4,5\}$.

\begin{figure}[h!t]
    \centering
    \includegraphics[width=8.4cm]{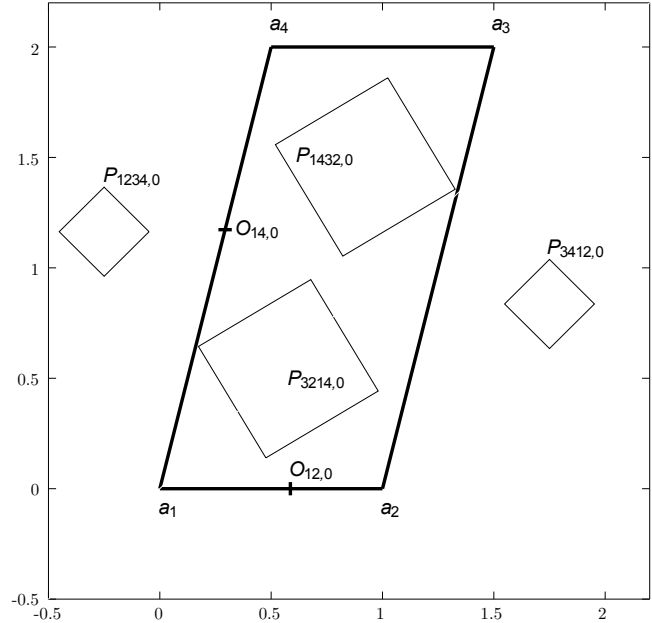}
    \caption{Example illustrating the theorem \ref{th_parallelogram2}. The coordinates of the vertices $a_1$, $a_2$, $a_3$ and $a_4$ of the parallelogram $P$ are respectively
$(0,0)$, $(1,0)$, $(3/2,2)$ and $(1/2,2)$. The parallelogram $P$, the squares $P_{1234,0}$, $P_{3214,0}$, $P_{3412,0}$, $P_{1432,0}$ and the points $O_{12,0}$, $O_{14,0}$ are shown.}
    \label{fig_square1}
\end{figure}

\begin{figure}[h!t]
    \centering
    \includegraphics[width=8.4cm]{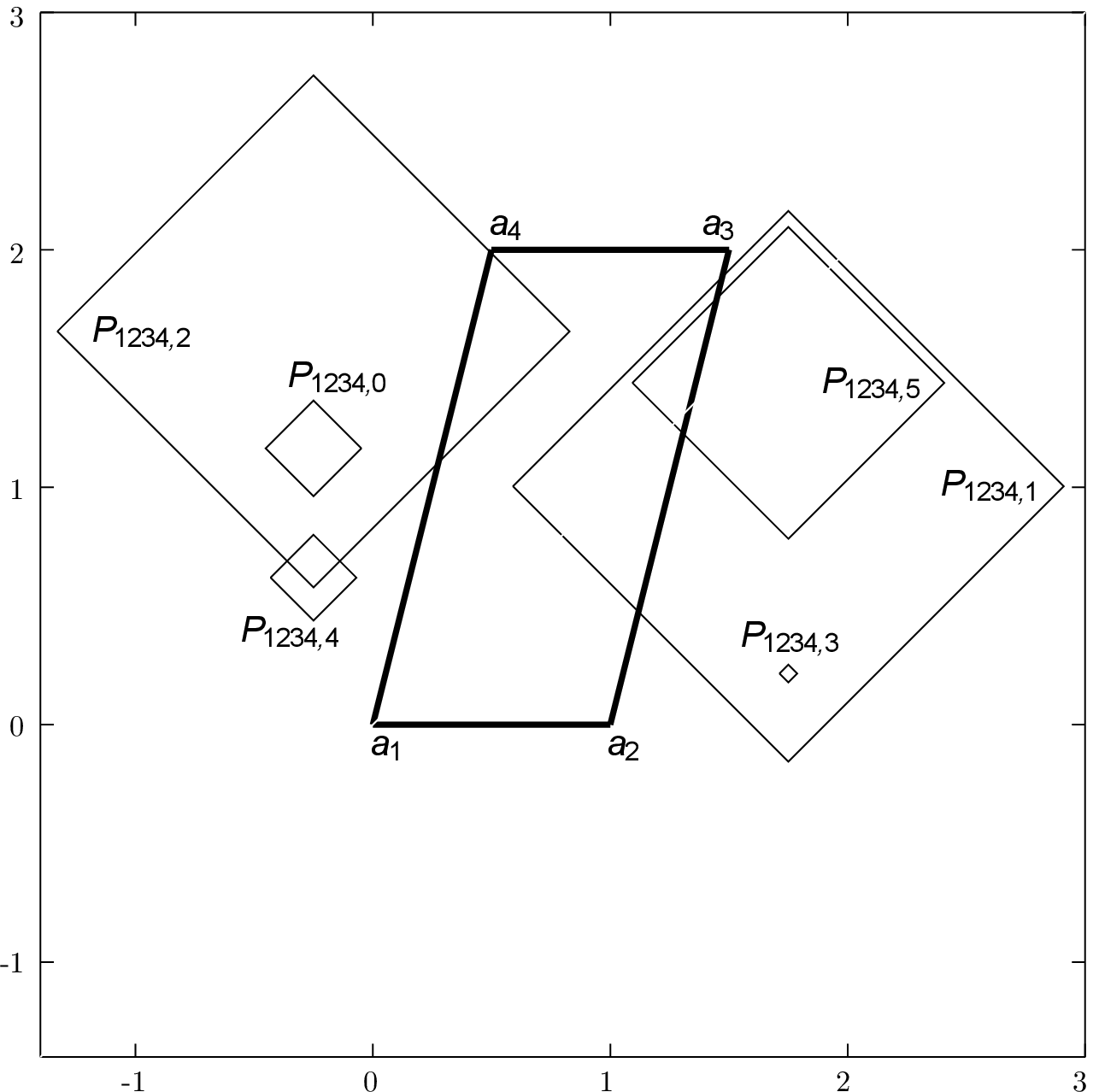}
    \caption{Example illustrating the theorem \ref{th_parallelogram2}. The coordinates of the vertices $a_1$, $a_2$, $a_3$ and $a_4$ of the parallelogram $P$ are respectively
$(0,0)$, $(1,0)$, $(3/2,2)$ and $(1/2,2)$. The parallelogram $P$ and the squares $P_{1234,0}$, $P_{1234,1}$, $P_{1234,2}$, $P_{1234,3}$, $P_{1234,4}$, $P_{1234,5}$ are shown.}
    \label{fig_square2}
\end{figure}

\section{Conclusion and outlook}
Starting from any quadrilateral $P$, the theorem \ref{th_parallelogram1} reveals essentially six families of parallelograms.
For each of these parallelograms, the theorem \ref{th_parallelogram2} reveals essentially four families of squares.

Point 1 of theorem \ref{th_parallelogram1} can be generalized to any polygon with an even number of edges.
For example, if $a_1$, $a_2$, $a_3$, $a_4$, $a_5$, $a_6$ denote the six points of a hexagon, $\hat{a}_i$ is the angle at the vertex $a_i$, $\alpha_{i,n}=\frac{2n+1}{4}\hat{a}_i$ for $n\in\Z$ and $r_i=r_{a_i,\alpha_{i,n}}$ so that $r_ir_jr_kr_lr_pr_q$ is a rotation through $(2n+1)\pi$ when $\{i,j,k,l,p,q\}=\{1,2,3,4,5,6\}$, then
\begin{align*}
&(r_1r_2r_3r_4r_5r_6)(r_1r_2r_3r_5r_6r_4)(r_2r_3r_1r_5r_6r_4)\cdots \\
&\quad(r_2r_3r_1r_6r_4r_5)(r_3r_1r_2r_6r_4r_5)(r_3r_1r_2r_4r_5r_6) \\
&=r_1r_2r_3r_4r_5r_6r_1r_2r_3(r_5r_6r_4r_2r_3r_1r_5r_6r_4r_2r_3r_1)\cdots \\
&\quad(r_6r_4r_5r_3r_1r_2r_6r_4r_5r_3r_1r_2)r_4r_5r_6 \\
&=id.
\end{align*}
This means that
\begin{equation*}
z_{b_{123456}}-z_{b_{123564}}+z_{b_{231564}}-z_{b_{231645}}+z_{b_{312645}}-z_{b_{312456}}=0
\end{equation*}
where $b_{ijklpq}$ is the fixed point of $r_ir_jr_kr_lr_pr_q$.
Hence, the hexagon joining the points $b_{123456}$, $b_{123564}$, $b_{231564}$, $b_{231645}$, $b_{312645}$ and $b_{312456}$ has not twelve degrees of freedom, but only ten.
We could for example inquire about a procedure which, when applied a sufficient number of times, eventually leads to a regular hexagon.

\bigskip

\subsection*{Acknowledgment}
The author gratefully aknowledges fruitful discussions with Oleg Ogievetsky.

\bigskip

\bigskip

\bigskip

INSTITUT PPRIME

CNRS - UNIVERSITE DE POITIERS - ISAE-ENSMA

SP2MI - 11 BOULEVARD MARIE ET PIERRE CURIE

F86962 FUTUROSCOPE CHASSENEUIL, FRANCE

\textit{E-mail address}: \texttt{pierre.godard@univ-poitiers.fr}\bigskip

\end{document}